\documentclass[mlq,fleqn]{w-art}
\usepackage{times}
\usepackage{amsmath}
\usepackage{w-thm}
\usepackage{bm}
\usepackage{times,cite,w-thm}
\theoremstyle{plain}

\theoremstyle{definition}

\newtheorem{question}{Question}
\usepackage[]{graphicx,amssymb}
\begin{document}
\DOIsuffix{Submitted October 20013}
\Volume{}
\Issue{}
\Month{}
\Year{2014}
\pagespan{1}{}
\Receiveddate{XXXX}
\Reviseddate{XXXX}
\Accepteddate{XXXX}
\Dateposted{XXXX}



\title[Dedekind Cuts]{Ramsey for $\mathcal{R}_{1}$ ultrafilter mappings and their Dedekind cuts}


\author[T. Trujillo]{Timothy Trujillo\inst{1}%
  \footnote{Corresponding author\quad E-mail:~\textsf{Timothy.Trujillo@du.edu},
 }}
\address[\inst{1}]{University of Denver, Department of Mathematics, 2360 S Gaylord St, Denver, CO 80208, USA}

\begin{abstract}
Associated to each ultrafilter $\mathcal{U}$ on $\omega$ and each map $p:\omega\rightarrow \omega$ is a Dedekind cut in the ultrapower $\omega^{\omega}/p( \mathcal{U})$. Blass has characterized, under CH, the cuts obtainable when $\mathcal{U}$ is taken to be either a p-point ultrafilter, a weakly-Ramsey ultrafilter or a Ramsey ultrafilter. 

Dobrinen and Todorcevic have introduced the topological Ramsey space $\mathcal{R}_{1}$. Associated to the space $\mathcal{R}_{1}$ is a notion of Ramsey ultrafilter for $\mathcal{R}_{1}$ generalizing the familiar notion of Ramsey ultrafilter on $\omega$.  We characterize, under CH, the cuts obtainable when $\mathcal{U}$ is taken to be a Ramsey for $\mathcal{R}_{1}$ ultrafilter and $p$ is taken to be any map. In particular, we show that the only cut obtainable is the standard cut, whose lower half consists of the collection of equivalence classes of constants maps. 

Forcing with $\mathcal{R}_{1}$ using almost-reduction adjoins an ultrafilter which is Ramsey for $\mathcal{R}_{1}$. For such ultrafilters $\mathcal{U}_{1}$, Dobrinen and Todorcevic have shown that the Rudin-Keisler types of the p-points within the Tukey type of $\mathcal{U}_{1}$ consists of a strictly increasing chain of rapid p-points of order type $\omega$. We show that for any Rudin-Keisler mapping between any two p-points within the Tukey type of $\mathcal{U}_{1}$ the only cut obtainable is the standard cut. These results imply existence
theorems for special kinds of ultrafilters.

\end{abstract}
\maketitle                   

\section{Introduction}
In this section, we define the notion of a Rudin-Keisler mapping and associate to each mapping a Dedekind cut. Then we state some results of Blass in \cite{BlassCut} characterizing, under CH, the types of cuts obtainable for Rudin-Keisler mappings from a p-point or a weakly-Ramsey ultrafilter on $\omega$. In last part of this section, we provide an outline of the rest of the article and highlight its main results.

We remind the reader of the {Rudin-Keisler reducibility relation.} If $\mathcal{U}$ is an ultrafilter on the base set $X$ and $\mathcal{V}$ is an ultrafilter on the base set $Y$, then we say that $\mathcal{V}$ is \emph{Rudin-Keisler reducible to} $\mathcal{U}$ and write $\mathcal{V} \le_{RK} \mathcal{U}$ if there there exists a function $f:X \rightarrow Y$ such that $\mathcal{V} = f( \mathcal{U})$, where
\begin{equation}
f(\mathcal{U}) = \left< \{ f(Z) : Z\in\mathcal{U}\} \right>.
\end{equation}
A \emph{Rudin-Keisler mapping from $\mathcal{U}$ to $\mathcal{V}$} is a function $f:X \rightarrow Y$ such that $\mathcal{V} = f( \mathcal{U})$.

Associated to each ultrafilter $\mathcal{U}$ on $X$ is an equivalence relation on $\omega^{X}$. If $f$ and $g$ are two functions from $X$ to $\omega$ then we say that \emph{$f$ and $g$ are equivalent mod $\mathcal{U}$} if there exists $Z\in \mathcal{U}$ such that $f\upharpoonright Z = g \upharpoonright Z$. The \emph{ultrapower} $\omega^{X}/\mathcal{U}$ is the collection of all equivalence classes with respect to this equivalence. All operations and relations defined on $\omega$ have natural extensions making the ultrapower an elementary extension of the standard model of $\omega$. In particular, $\omega^{X}/\mathcal{U}$ forms a linearly ordered set. ( In this case, $[f]\le[g]$ if and only if $\{x\in X: f(x) \le g(x)\}\in\mathcal{U}$.)

Recall that, a \emph{Dedekind cut of a linearly ordered set} is a partition $(S,L)$ of the linear order such that no element of $L$ precedes any element of $S$. We follow the work of Blass in \cite{BlassCut} and associate to each Rudin-Keisler mapping from $\mathcal{U}$ on $X$ to $\mathcal{V}$ on $Y$ a Dedekind cut in the ultrapower $\omega^{Y}/\mathcal{V}$. A cut $(S,L)$ in the ultrapower is said to be \emph{proper} if $L$ is nonempty and $S$ contains the equivalence class of each constant map. The cut given by taking $S$ to be the set of equivalence classes of constant maps is called \emph{the standard cut}. 
\begin{definition}[ \cite{BlassCut}]
Let ${\cal U}$ be an ultrafilter on the base set $X$ and $p: X \rightarrow Y$. For any $A \subseteq X$, we define \emph{the cardinality function of $A$ relative to $p$} by
\begin{equation}
C_{A}(y) = | A \cap p^{-1}\{y\}|\ \mbox{ for } y\in Y.
\end{equation}
The set of all equivalence classes of cardinality functions of sets in ${\cal U}$, and all larger elements of $\omega^{X} /p( {\cal U})$, constitute the upper part $L$ of a cut $(S,L)$ of $\omega^{X} / p( {\cal U} )$, which we call \emph{the cut associated to $p$ and ${\cal U}$}. (If $C_{A}(n)$ is infinite for some $y$ then $C_{A}\not \in \omega^{Y}$ and has no equivalence class, so it makes no contribution to $L$; it is entirely possible for $L$ to be empty.) 
\end{definition}

 The cut associated to $p$ and ${\cal U}$ is proper if and only if $p$ is finite-to-one but not one-to-one on any set in ${\cal U}$. Additionally, the existence of a proper cut in $\omega^{X}/{\cal U}$ implies that ${\cal U}$ is non-principal. The next three theorems are due to Blass and appear as Theorems 1, 2 and 4 in \cite{BlassCut}. The first theorem shows that certain Dedekind cuts are not obtainable from Rudin-Keisler mappings. In the remaining theorems of this section, $(S,L)$ is assumed to be a proper cut.
\begin{theorem}[ \cite{BlassCut}]
\label{cut-thm1}
$(S,L)$ is the cut associated to some map of some ultrafilter to $\mathcal{ U}$ if and only if $S$ is closed under addition. 
\end{theorem}
 Before stating the next two theorems we remind the reader of the definitions of some special types of ultrafilters on $\omega$.
\begin{definition}Let $\mathcal{U}$ be an ultrafilter on $\omega$.
\begin{enumerate}
 \item $\mathcal{U}$ is a \emph{p-point ultrafilter}, if for each sequence $A_{0}\supseteq A_{1} \supseteq A_{2}\supseteq \cdots$ of members of $\mathcal{U}$ there exists $A \in \mathcal{U}$ such that for each $i<\omega$, $A \subseteq^{*} A_{i}$. (Here $\subseteq^{*}$ denotes the almost-inclusion relation.)

\item $\mathcal{U}$ is a \emph{weakly-Ramsey ultrafilter}, if for each partition of the two-element subsets of $\omega$ into three parts there exists an element of $\mathcal{U}$ all of whose two-element subsets lie in two parts of the partition. 

\item $\mathcal{U}$ is a \emph{Ramsey ultrafilter}, if for each partition of the two-element subsets of $\omega$ into two parts there exists an element of $\mathcal{U}$ all of whose two-element subsets lie in exactly one part of the partition. 
\end{enumerate}
\end{definition}
\begin{remark}
It clear that every Ramsey ultrafilter is weakly-Ramsey. In Theorem 5 of \cite{BlassCut}, Blass has shown that every weakly-Ramsey ultrafilter is a p-point.
\end{remark}
\begin{theorem}[ \cite{BlassCut}]
\label{cut-thm2}
Assume CH. $(S,L)$ is the cut associated to some map of some p-point ultrafilter to ${\cal U}$ if and only if ${\cal U}$ is a p-point, $S$ is closed under addition, and every countable subset of $L$ has a lower bound in $L$.
\end{theorem}
\begin{theorem}[\cite{BlassCut}]
\label{cut-thm4}
Assume CH. $(S,L)$ is the cut associated to some map of some weakly Ramsey ultrafilter to ${\cal U}$ if and only if ${\cal U}$ is Ramsey, $S$ is closed under exponentiation, and every countable subset of $L$ has a lower bound in $L$.  
\end{theorem}

Blass has remarked in \cite{BlassCut} that many of the ultrafilter-theoretic concepts involved in the previous theorems have natural model-theoretic interpretations in terms of ultrapowers. Following Blass, we consider models that are elementary extensions of the standard model whose universe is $\omega$ and whose relations and functions are all the relations and functions on $\omega$. Suppose that $N$ is such a model. An element $x\in N$ is said to \emph{generate $N$ over the submodel $M$} if and only if no proper submodel of $N$ includes $M \cup \{x\}$.

In \cite{BlassCut}, Blass notes that, if $f:X \rightarrow Y$ and $\mathcal{U}$ is an ultrafilter on $X$, then
\begin{equation}
f^{*}: \omega^{Y}/ f(\mathcal{U}) \rightarrow \omega^{X}/ \mathcal{U}
\end{equation}
is an elementary embedding. Furthermore, $f$ is an isomorphism of ultrafilters if and only if $f^{*}$ is an isomorphism of models. The image of $f^{*}$ is cofinal in $\omega^{Y}/\mathcal{U}$ if and only if $f$ is finite-to-one on some set of $\mathcal{U}$. Hence, $\mathcal{U}$ is a p-point (Ramsey ultrafilter) if and only if every nonstandard submodel of $\omega^{X}/\mathcal{V}$ is cofinal in ( equal to ) $\omega^{X}/\mathcal{U}$. An element $x \in \omega^{Y}/f(\mathcal{V})$ is in the upper half of the cut associated to $f$ and $\mathcal{U}$ if and only if $f^{*}(x)$ is greater than some generator of $\omega^{X}/\mathcal{U}$ over $f^{*}( \omega^{Y}/f(\mathcal{U}))$.

\begin{remark}\label{IsomorphismRemark}
Let ${\cal U}$ be an ultrafilter on the base set $X$ and $p: X \rightarrow Y$. Suppose that $g:W \rightarrow X$ is a bijection and $\mathcal{W}$ is an ultrafilter on $W$ such that $g(\mathcal{W})= \mathcal{U}$. Then the cut associated to $p \circ g$ and $\mathcal{W}$ is exactly the cut associated to $p$ and $\mathcal{U}$. Additionally, suppose that $h: Y \rightarrow Z$ is a bijection. Since $h^{*}$ is an isomorphism of models, it follows that if $(S,L)$ is the cut associated to $h \circ p$ and $\mathcal{U}$ then $(h^{*\prime \prime}S, h^{* \prime \prime}L)$ is the cut associated to $p$ and $\mathcal{U}$. In particular, if $(S,L)$ is the standard cut in $\omega^{\omega}$ then $(h^{*\prime \prime}S, h^{*\prime \prime}L)$ is the standard cut in $\omega^{\omega}/ p(\mathcal{U})$.
\end{remark}

The purpose of this paper is to prove analogous results for ultrafilters satisfying similar properties. In Section $\ref{section2}$ we introduce the setting for the main results of this article, namely the topological Ramsey space $\mathcal{R}_{1}$. In Section $\ref{section3}$ we introduce the generalization of the notion Ramsey ultrafilter we study in later sections. In Section $\ref{section4}$, we characterize the cuts obtainable, under CH, from the generalization of Ramsey we defined in Section $\ref{section3}$. The next theorem which we prove in Section $\ref{section4}$ is one of the two main results of this manuscript.
\begin{theorem}\label{MainResult1}
Assume CH. $(S,L)$ is the cut associated to some map of some Ramsey for $\mathcal{R}_{1}$ ultrafilter on $[T_{1}]$ to ${\cal V}$ if and only if ${\cal V}$ is selective and $(S,L)$ is the standard cut in $\omega^{\omega}/{\cal V}$.
\end{theorem}

 In Section $\ref{section5},$ we introduce the basic definitions associated with the Tukey theory of ultrafilters. Applying theorems of Dobrinen and Todorcevic in \cite{Ramsey-Class} about ultrafilters generated from generic subsets of $(\mathcal{R}_{1}, \le^{*})$ we prove the second main result of this paper. 
\begin{theorem}\label{MainResult2}
Suppose $\mathcal{U}_{1}$ is a Ramsey for $\mathcal{R}_{1}$ ultrafilter on $T_{1}$ generated by a generic subset of $(\mathcal{R}_{1}, \le^{*})$. If $(S,L)$ is the cut associated to some map from some p-point ultrafilter in the Tukey type of $\mathcal{U}_{1}$ to some ultrafilter $\mathcal{V}$ then $\mathcal{V}$ is a p-point ultrafilter and $(S,L)$ is the standard cut.
\end{theorem}
In Section $\ref{section6}$ we show that the main results imply the existence of special ultrafilters. We then conclude with some questions about the types of cuts obtainable from ultrafilters defined from other similar topological Ramsey spaces.

The author would like to express his deepest gratitude to Natasha Dobrinen for valuable comments and suggestions that helped make this article and its proofs more readable.

\section{ The topological Ramsey space $\mathcal{R}_{1}$}\label{section2}

We begin this section with the definition given by Dobrinen and Todorcevic in \cite{Ramsey-Class} of the triple $(\mathcal{R}_{1}, \le, r)$. The construction of $\mathcal{R}_{1}$ was motivated by the work of Laflamme in \cite{Laflamme} which uses forcing to adjoin a weakly-Ramsey ultrafilter satisfying complete combinatorics over $\mathrm{HOD}(\mathbb{R})$ (see \cite{Laflamme} for a definition of complete combinatorics.)
\begin{definition}[$(\mathcal{R}_{1}, \le, r)$, \cite{Ramsey-Class}]
For each $i<\omega$, let
\begin{equation}
T_{1}(i) = \left  \{ \left< \ \right>, \left < i \right >, \left< i, j \right >:  i(i+1)\le 2j <(i+1)(i+2) \right\} \mbox{ and }
\end{equation}
\begin{equation}
T_{1} = \bigcup_{i<\omega} T_{1}(i).
\end{equation}
\begin{figure}[!h]
\includegraphics[width=\linewidth]{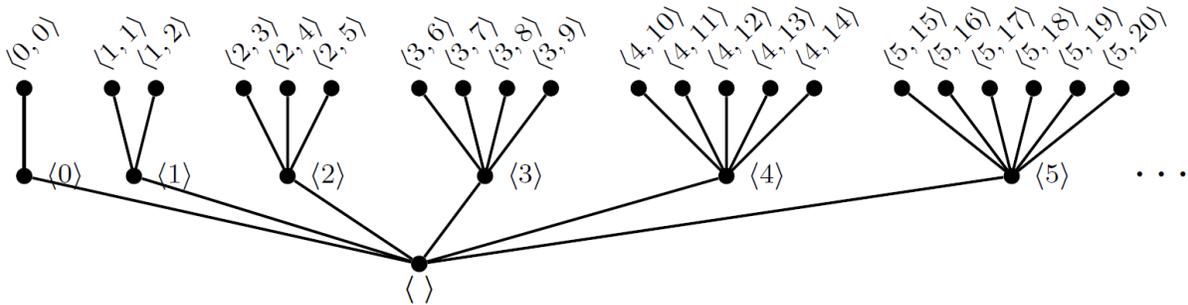}
\caption{Graph of $T_{1}$}
\label{fig:1}
\end{figure}
$\mathcal{R}_{1}$ consists of all subtrees of $T_{1}$ which are isomorphic to $T_{1}$. More precisely, $S \subseteq T_{1}$ is a member of $\mathcal{R}_{1}$ if and only if there exists a strictly increasing sequence $(k_{i})_{i<\omega}$ of natural numbers such that the following conditions hold:
\begin{align}
& \mbox{For each $i<\omega$, $S \cap T_{1}(k_{i})$ is isomorphic to $T_{1}(i)$.}\\
& \mbox{For each $k<\omega$, $S \cap T_{1}(k)\not = \{ \left< \ \right>\}$ implies that $k=k_{i}$ for some $i<\omega$.}
\end{align}

 For each tree $S \in \mathcal{R}_{1}$ and $i<\omega$, let $S(i) = S\cap T(k_{i})$ and $r_{i}(S) = \bigcup_{j<i} S(j)$. Let $\mathcal{AR} = \bigcup_{i<\omega} \{ r_{i}(S) : S\in\mathcal{R}_{1}\}$ and define $r: \omega\times \mathcal{R}_{1} \rightarrow \mathcal{AR}$ by letting $r(i, S)=r_{i}(S)$. For each $i<\omega$, let $\mathcal{R}(i) = \{ S(i) : S \in \mathcal{R}_{i}\}$ and $\mathcal{AR}_{i} = \{r_{i}(S): S\in\mathcal{R}_{1}\}$.

For $S,T\in\mathcal{R}_{1}$, $S\le T$ if and only if there exists a strictly increasing sequence $(k_{i})_{i<\omega}$ of natural numbers such that for each $i<\omega$, $S(k_{i}) \subseteq T(i)$. For $S,T\in\mathcal{R}_{1}$, $S\le^{*} T$ if and only there exists $i<\omega$ such that $S\setminus r_{i}(S)\subseteq T$. The relation $\le^{*}$ is called the \emph{almost-reduction relation on $\mathcal{R}_{1}$}.
\end{definition}
\begin{notation} If $S\in\mathcal{R}_{1}$ then we let $[S]$ denote the maximal nodes of $S$, that is the length two sequences in $S$. Similarly, for $s\in \mathcal{AR}$ and $p\in \bigcup_{i<\omega}\mathcal{R}_{1}(i)$ we let $[s]$ and $[p]$ denote the collection of maximal nodes of $s$ and $p$ respectively. If $S\in\mathcal{R}_{1}$ and $i<\omega$ then we let, $\mathcal{R}_{1}(i)\upharpoonright S = \{ p\in \mathcal{R}_{1}(i) : p \subseteq S\}$ and $\mathcal{AR}_{i}|S = \{ s \in \mathcal{AR}_{i} : s \subseteq S\}.$ For each $s\in\mathcal{AR}$, we let $\mathrm{depth}_{T_{1}}(s)$ be the least natural number $n$ such that $s \subseteq r_{n}(T_{1})$.
\end{notation}

Topological Ramsey theory was initiated by Ellentuck in \cite{Ellentuck} with his proof of an infinite-dimensional analogue of the Ramsey theorem called the Ellentuck theorem. After similar examples where discovered, the abstract notion of a topological Ramsey space was introduced in by Carlson and Simpson in \cite{CarlsonSimpson}. Carlson and Simpson also introduced a finite set of axioms sufficient for proving an abstract version of the Ellentuck theorem. The text \cite{RamseySpaces} by Todorcevic is now the standard reference for topological Ramsey spaces and proves the abstract Ellentuck theorem using four axioms. 

\begin{definition}
For each $s\in \mathcal{AR}$ and $S\in \mathcal{R}_{1}$ we let $[s,S]=\{ T\in \mathcal{R}_{1}: (\exists n) \ s = r_{n}( T ) \ \& \ T \le S\}$.
The \emph{Ellentuck topology on $\mathcal{R}_{1}$} is the topology generated by $\{[s,S]: s\in \mathcal{AR} \ \& \ S \in \mathcal{R}_{1}\}$. A subset ${\mathcal X}$ of ${\mathcal R}_{1}$ is \emph{Ramsey} if for every $\emptyset \not= [s,S],$ there is a $T \in [s,S]$ such that $[s,T] \subseteq {\mathcal X}$ or $[s,T]\cap {\mathcal X} = \emptyset.$
A subset ${\mathcal X}$ of ${\mathcal R}$ is \emph{Ramsey null} if for every $\emptyset \not= [s,S],$ there is a $T \in [s,S]$ such that $[s,S]\cap {\mathcal X} = \emptyset.$
\end{definition}
\begin{theorem}[The abstract Ellentuck theorem for $\mathcal{R}_{1}$, \cite{Ramsey-Class}] 
The triple $(\mathcal{R}_{1}, \le, r )$ forms a \emph{topological Ramsey space}. That is, every subset of $\mathcal{ R}_{1}$ with the Baire property is Ramsey and every meager subset of $\mathcal{R}_{1}$ is Ramsey null.
\end{theorem}
The next theorem, which we use in Section $\ref{section6}$, can be thought of as a finite version of the Ramsey theorem for the space $\mathcal{R}_{1}$. The result is Theorem 3.5 of Mijares in \cite{MijaresGalvin} applied to the topological Ramsey space $\mathcal{R}_{1}$.
\begin{theorem}[The finite Ramsey theorem for $\mathcal{R}_{1}$, \cite{MijaresGalvin}]\label{finiteRamseyTheorem}
Let $k,n<\omega$ with $k\le n$ be given. Then, there exists $m<\omega$ such that for each $p\in\mathcal{R}_{1}(m)$ and each partition of $\mathcal{R}_{1}(k)\upharpoonright p$ into two parts there exists $q\in\mathcal{R}_{1}(n)\upharpoonright p$ such that $\mathcal{R}_{1}(k)\upharpoonright q$ lies in exactly one one part of the partition.
\end{theorem}

Next we define a collection of natural projection maps related to the space $\mathcal{R}_{1}$. Dobrinen and Todorcevic in \cite{Ramsey-Class} have used these projections to completely characterize the Tukey ordering below weakly-Ramsey ultrafilters obtained from forcing with $\mathcal{R}_{1}$ using almost-reduction.
\begin{definition} Each $x \in[T_{1}]$ is a sequence of natural numbers of length two, which we denote by $\left<x_{0},x_{1}\right>$.
Let $\pi:[T_{1}] \rightarrow \omega$ be the map which sends $x\in[T_{1}]$ to $x_{0}$.  For each $i<j<\omega$, define $\pi_{T(i)}:\mathcal{R}_{1}(j) \rightarrow \mathcal{R}_{1}(i)$ to be the map that removes the right-most $j-i$ branches of a given element of $\mathcal{R}_{1}(j)$.
\end{definition}

\section{Selective and Ramsey for $\mathcal{R}_{1}$.} \label{section3}
In this section we introduce the generalizations of Ramsey ultrafilter that we study in this manuscript and prove some theorems needed to for the main results in Sections $\ref{section5}$ and $\ref{section6}$. First we remind the reader of the definition of  selective ultrafilter on $\omega$.
\begin{definition} Let $\mathcal{U}$ be an ultrafilter on $\omega$. $\mathcal{U}$ is \emph{selective}, if for each sequence $A_{0}\supseteq A_{1} \supseteq A_{2}\supseteq \cdots$ of members of $\mathcal{U}$ there exists $A=\{a_{0},a_{1},\dots\} \in \mathcal{U}$ enumerated in increasing order such that for each $i<\omega$, $A\setminus\{a_{0}, \dots, a_{i-1}\} \subseteq A_{i}$. 
\end{definition}
By a result of Kunen in \cite{Booth}, an ultrafilter is Ramsey if and only if it is selective. For each topological Ramsey space $\mathcal{R}$, Mijares in \cite{MijaresSelective} has introduced the notions of Ramsey for $\mathcal{R}$ and selective for $\mathcal{R}$ ultrafilters. If $\mathcal{R}$ is taken to be the Ellentuck space the notions are equivalent and reduce to Ramsey and selective given above. Below, we follow the presentation of Ramsey for $\mathcal{R}_{1}$ and selective for $\mathcal{R}_{1}$ given by Dobrinen and Todorcevic in \cite{Ramsey-Class}.

\begin{definition}Let $\mathcal{U}$ be an ultrafilter on $[T_{1}]$.  ${\mathcal U}$ is \emph{generated by ${\mathcal C}\subseteq {\mathcal R}_{1}$} if and only if $\{ [S] : S\in{\mathcal C}\}$ is cofinal in $({\mathcal U}, \supseteq)$. An ultrafilter $\mathcal{U}$ generated by ${\mathcal C}\subseteq \mathcal{R}_{1}$ is \emph{selective for ${\mathcal R}_{1}$}, if for each decreasing sequence $S_{0}\ge S_{1} \ge S_{2} \ge \cdots$ of members of ${\mathcal C}$, there is another $S \in {\mathcal C}$ such that for each $i<\omega$, $S\setminus r_{i}(S) \subseteq S_{i}.$ 
\end{definition}

\begin{theorem}\label{selectThm}
Suppose that $\mathcal{U}$ is an ultrafilter on the base set $[T_{1}]$ and generated by a subset ${\cal C}$ of ${\cal R}_{1}$. The following statements are equivalent:
\begin{enumerate}
\item $\mathcal{U}$ is selective for $\mathcal{R}_{1}$.
\item $\mathcal{U}$ is a p-point and $\pi(\mathcal{U})$ is selective. 
\item for each decreasing sequence $S_{0}\ge S_{1} \ge S_{2} \ge \cdots$ of members of ${\mathcal C}$, there is another $S \in {\mathcal C}$ such that for each $n<\omega$, $S\setminus r_{n}(S) \subseteq S_{\mathrm{depth}_{T_{1}}(r_{n}(S))}.$ 
\end{enumerate}
\end{theorem}
\begin{proof} First we show that $1. \Rightarrow 2.$ Suppose that $\mathcal{U}$ is selective for $\mathcal{R}_{1}$.  To show that $\mathcal{U}$ is a p-point consider a sequence $A_{0}\supseteq A_{1} \supseteq A_{2} \supseteq \dots$ of elements of $\mathcal{U}$. Since $\mathcal{U}$ is generated by $\mathcal{C}$ there exists a sequence $S_{0} \ge S_{1} \ge S_{2} \ge \dots$ of elements of $\mathcal{C}$ such that for all $i<\omega$, $[S_{i}] \subseteq A_{i}$. Since $\mathcal{U}$ is selective for $\mathcal{R}_{1}$ there exists $S\in\mathcal{S}$ such that for all $i<\omega$, $S\setminus r_{i}(S) \subseteq S_{i}$. As each $r_{i}(S)$ is finite for each $i<\omega$, $[S] \subseteq^{*} [S_{i}]\subseteq A_{i}$. Hence $\mathcal{U}$ is a p-point ultrafilter on $[T_{1}]$.

To show that $\pi(\mathcal{U})$ is selective consider a sequence $X_{0}\supseteq X_{1} \supseteq X_{2} \supseteq \dots$ of elements of $\pi(\mathcal{U})$. Since $\mathcal{U}$ is generated by $\mathcal{C}$ there exists a sequence $S_{0} \ge S_{1} \ge S_{2} \ge \dots$ of elements of $\mathcal{C}$ such that for all $i<\omega$, $\pi'' [S_{i}] \subseteq X_{i}$.  Since $\mathcal{U}$ is selective for $\mathcal{R}_{1}$ there exists $S\in\mathcal{S}$ such that for all $i<\omega$, $S\setminus r_{i}(S) \subseteq S_{i}$. Let $\{x_{0}, x_{1}, \dots\}$ be the increasing enumeration of $\pi''[S]$. Then for each $i<\omega$, $\pi''[S] \setminus\{x_{0}, x_{1}, \dots, x_{i-1}\} = p''([S]\setminus [r_{i}(S)]) \subseteq \pi''[S_{i}] \subseteq X_{i}$. $\pi(\mathcal{U})$ is selective as $\{x_{0},x_{1},\dots\}$ is in $\pi(\mathcal{U})$.

Next we show that $2.\Rightarrow 3.$ Suppose that ${\cal U}$ is a p-point and $\pi({\cal U})$ is selective. To show that $\mathcal{U}$ satisfies condition $3.$, consider an arbitrary decreasing sequence $S_{0}\ge S_{1} \ge S_{2} \ge \cdots$  of members of ${\cal C}$. There is an $S\in {\cal C}$ such that for each $n<\omega$, $[S] \subseteq^{*} [S_{n}]$. Let $(k_{i})_{i<\omega}$ be the strictly increasing sequence such that for all $i<\omega$, $S(i) \subseteq T_{1}(k_{i})$. Define $(k'_{n})_{n<\omega}$ recursively by letting,
\begin{equation}
\begin{cases}
k'_{0} = k_{0},\\
\mbox{$k'_{i+1}$ is the smallest $k_{j}>k'_{i}$ such that $\{ x\in [S]: \pi(x) >k_{j}\} \subseteq X_{k'_{i}}$.}
\end{cases}\end{equation}
Now define $g:\omega \rightarrow \omega$ by letting
\begin{equation}
g(n)= i  \text{ if }  k'_{i} \le n < k'_{i+1}.
\end{equation}
Note that $g$ can not be constant mod $\pi({\cal U})$ as $\pi(\mathcal{U})$ is non-principal. Since $\pi({\cal U})$ is selective, it must be the case that there is a $Y \subseteq \pi''[S]$ such that $g$ is increasing on $Y\in \pi({\cal U})$. Enumerate $Y$ in increasing order,  as $\{y_{0},y_{1}, \dots \}$. Then either $\{y_{0},y_{2}, y_{4}, \dots \}$ or $\{y_{1}, y_{3},y_{5}, \dots \}$ is a member of $\pi({\cal U})$. Let $Z=\{z_{0}, z_{1}, \dots\}$ denote which ever is in $\pi({\cal U})$. By construction, we find that for each pair $i<j$ of natural numbers there exists $k'_{l+1}<\omega$ such that $z_{i} < k'_{l+1}<z_{j}$. 

Since $\pi^{-1}(Z) \in {\cal U}$ and $[S] \in {\cal U}$, there is a $S' \in {\cal C}$ such that $[S']\subseteq \pi^{-1}(Z) \cap [S]$. Let $(k''_{i})_{i<\omega}$ be the strictly increasing sequence such that for all $i<\omega$, $S'(i) \subseteq T_{1}(k''_{i})$.  Since $\pi''[S'] \subseteq Z$, we find that for each $n<\omega$, each $m>n$ and each $s \in [S'(m)]$, there exists $k'_{l+1}$ such that $\pi(s) = k''_{m}> k'_{l+1} > k''_{n}.$ 
By definition of the sequence $(k'_{i})_{i<\omega}$, it follows that $s \in [S_{k'_{l}}]$. On the other hand,  $ k'_{l+1} > k''_{n}$ implies that $ S_{k'_{l}} \subseteq S_{k''_{n}}$. So $s \in [S_{k''_{n}}]$ and $ S'\setminus r_{n}(S') \subseteq S_{k''_{n}}= S_{\mathrm{depth}_{T_{1}}(r_{n}(S'))}$. Hence $2. \Rightarrow 3.$ holds.

Next note that for each $S\in\mathcal{R}_{1}$, $\mathrm{depth}_{T_{1}}(r_{n}(S)) \ge n$. Hence $3.\Rightarrow 1.$ holds trivially. 
\end{proof}

\begin{definition}Let $\mathcal{U}$ be an ultrafilter on $[T_{1}]$ generated by ${\mathcal C}\subseteq \mathcal{R}_{1}.$ $\mathcal{U}$ is \emph{Ramsey for ${\mathcal R}_{1}$}, if for every $i<\omega$ and every partition of $\mathcal{AR}_{i}$ into two parts there exists $S\in \mathcal{C}$ such that $\mathcal{AR}_{i}|S$ lies in one part of the partition.
\end{definition}

\begin{theorem}\label{EquivalentRamsey}
Suppose that $\mathcal{U}$ is an ultrafilter on $[T_{1}]$ generated by $\mathcal{C}\subseteq \mathcal{R}_{1}$. $\mathcal{U}$ is Ramsey for $\mathcal{R}_{1}$ if and only if $\mathcal{U}$ is selective for $\mathcal{R}_{1}$ and for each $n<\omega$, $\mathcal{U}|  \mathcal{R}_{1}(n) =\{ \mathcal{R}_{1}(n) \upharpoonright A : A\in \mathcal{C}\}$ forms an ultrafilter on $\mathcal{R}_{1}(n)$.
\end{theorem}
\begin{proof}
($\Rightarrow$) By a Lemma 3.8 of Mijares in \cite{MijaresSelective}, every Ramsey for $\mathcal{R}_{1}$ ultrafilter is selective for $\mathcal{R}_{1}$. One the other hand, if $\mathcal{U}$ is Ramsey for $\mathcal{R}_{1}$ then for each $n<\omega$, $\mathcal{U}|  \mathcal{R}_{1}(n) =\{ \mathcal{R}_{1}(n) \upharpoonright A : A\in \mathcal{C}\}$ forms an ultrafilter on $\mathcal{R}_{1}(n)$.

($\Leftarrow$) Suppose $\mathcal{U}$ is selective for $\mathcal{R}_{1}$ and for each $n<\omega$, $\mathcal{U}|  \mathcal{R}_{1}(n) =\{ \mathcal{R}_{1}(n) \upharpoonright S : S\in \mathcal{C}\}$ forms an ultrafilter on $\mathcal{R}_{1}(n)$. Since $\mathcal{AR}_{1} = \mathcal{R}_{1}(0)$ and $\mathcal{U}| \mathcal{R}_{1}(0)$ forms an ultrafilter on $\mathcal{R}_{1}(0)$, it follows that   if $i=0$ then every partition of $\mathcal{AR}^{1}_{i}$ into two parts there exists $S\in \mathcal{C}$ such that $\mathcal{AR}_{i}|S$ lies in one part of the partition. We proceed by induction on $i$ to show that $\mathcal{U}$ is Ramsey for $\mathcal{R}_{1}$. The previous remarks show that the base case of the induction holds. 

Let $i$ be a natural number and suppose that every partition of $\mathcal{AR}_{i}$ into two parts there exists $S\in \mathcal{C}$ such that $\mathcal{AR}_{i}|S$ lies in one part of the partition. Let $\{\Pi_{0}, \Pi_{1}\}$ be a partition of $\mathcal{AR}_{i+1}$. We show that there exists $S\in \mathcal{C}$ such that $\mathcal{AR}_{i+1}|S$ lies in one part of the partition.

For each $s\in \mathcal{AR}_{i}$, let $A_{s} = \{ p \in \mathcal{R}_{1}(i) : s \cup p \in \Pi_{0}\}$. Let
\begin{equation*}
\Pi_{0}' = \{ s \in \mathcal{AR}_{i} : A_{s} \in \mathcal{U}|\mathcal{R}_{1}(i)\} \mbox{ and } \Pi_{1}' = \{ s \in \mathcal{AR}_{i} : \mathcal{R}_{1}(i) \setminus A_{s} \in \mathcal{U}|\mathcal{R}_{1}(i)\}.
\end{equation*}
Since $\mathcal{U}|\mathcal{R}_{1}(i)$ forms an ultrafilter on $\mathcal{R}_{1}(i)$ it follows that $\{\Pi_{0}',\Pi_{1}'\}$ is a partition of $\mathcal{AR}_{i}$. By the inductive hypothesis, there exists $S\in \mathcal{C}$ and $j<2$ such that $\mathcal{R}_{1}(i) | S \subseteq \Pi_{j}'$. 

We first consider the case when $j=0$. In particular, for each $s\in \mathcal{AR}_{i}|S$, $A_{s} \in \mathcal{U}|\mathcal{R}_{1}(i)$. For each $n<\omega$, let $B_{n} = \bigcap_{\mathrm{depth}_{T_{1}}(s) \le n} A_{s} \in \mathcal{U}|\mathcal{R}_{1}$. Hence there exists a sequence $\{S_{n}: n<\omega\}$ of elements of $\mathcal{C}$ such that $S_{0} \ge S_{1} \ge S_{2} \ge \dots$ and for each $n<\omega$, $\mathcal{R}_{1}(i)|S_{n} \subseteq B_{n}$. By Theorem $\ref{selectThm}$ there exist $S\in \mathcal{C}$ such that for each $n<\omega$, $S\setminus r_{n}(S) \subseteq S_{\mathrm{depth}_{T_{1}}(r_{n}(S))}$.

Suppose that $t\in\mathcal{AR}_{i+1}|S$. Then $r_{i}(t) \in \mathcal{AR}_{i}|A$ and $t(i)\in \mathcal{R}_{1}(i)$. If $k= \mathrm{depth}_{S}(r_{i}(t))$ then $t(i) \in \mathcal{R}_{1}(i) |( S \setminus r_{k}(S)) \subseteq \mathcal{R}_{1}(i)|S_{\mathrm{depth}_{T_{1}}(r_{k}(S))}\subseteq A_{r_{i}(t)}$. Hence, $r_{i}(t) \cup t(i) \in \Pi_{0}$. So in the case when $j=0$, $\mathcal{AR}_{i+1}|S \subseteq \Pi_{j}$. By an identical argument in the case when $j=1$, there exists $S\in\mathcal{C}$ such that  $\mathcal{AR}_{i+1}|S \subseteq \Pi_{j}$. 

By induction we find that for each $i<\omega$ and each partition of $\mathcal{AR}^{1}_{i}$ into two parts there exists $S\in \mathcal{C}$ such that $\mathcal{AR}^{1}_{i}|S=\{ s\in\mathcal{AR}^{1}_{i}: s\subseteq S\}$ lies in one part of the partition. In other words, $\mathcal{U}$ is a Ramsey for $\mathcal{R}_{1}$ ultrafilter on $[T_{1}]$.
\end{proof}

\section{Ramsey for $\mathcal{R}_{1}$ ultrafilters and their Dedekind cuts} \label{section4}
In this section, assuming CH, we characterize the types of proper Dedekind cuts that can be obtained from a map $p:[T_{1}] \rightarrow \omega$ and a Ramsey for $\mathcal{R}_{1}$ ultrafilter. In the following theorems, all cuts are assumed to be proper.
\begin{lemma}\label{CUT1}
Let $\mathcal{U}$ be a Ramsey for $\mathcal{R}_{1}$ ultrafilter on $[{T}_{1}]$ generated by ${\cal C}\subseteq \mathcal{R}_{1}$ and $p$ be a map from $[T_{1}]$ to $\omega$.  The cut associated to $p$ and $\mathcal{U}$ is the standard cut in $\omega^{\omega} / p(\mathcal{U})$. \end{lemma}
\begin{proof}
Suppose that $(S,L)$ is the cut associated to $p$ and ${\cal U}$. Let $f:\omega\rightarrow \omega$ be given and suppose that $f$ is not constant mod $p({\cal U})$. For each $s\in\mathcal{AR}_{2}$ let $\{s_{0},s_{1}, s_{2}\}$ be the lexicographically increasing enumeration of $[s]$. Let $\{\Pi_{0}, \Pi_{1}, \Pi_{2}\}$ be the partition of $\mathcal{AR}_{2}$ given by letting
\begin{align}
\Pi_{0}&= \{ s \in\mathcal{AR}_{n} : p(s_{0}) < p(s_{1})\ \& \ p(s_{1})=p(s_{2})\},\\
\Pi_{1}&= \{ s \in\mathcal{AR}_{n} : p(s_{0}) < p(s_{1}) \ \& \  p(s_{1})\not=p(s_{2})\} \mbox{ and }\\
\Pi_{2}&= \{ s \in\mathcal{AR}_{n} :p(s_{0})\ge p(s_{1})\}.
\end{align}
Since $\mathcal{R}_{1}$ is Ramsey for $\mathcal{R}_{1}$ there exists $S'\in\mathcal{C}$ and $j<3$ such that $ \mathcal{AR}_{2}| S' \subseteq \Pi_{j}$. If $j=2$ then $p$ is bounded by $p(x)$ mod $\mathcal{U}$ where $x$ is the lexicographically least element of $[S']$. So if $j=2$ then $p$ is constant mod $\mathcal{U}$ and $(S,L)$ is not a proper cut. If $j=1$ then $p$ is one-to-one mod $\mathcal{U}$, so $(S,L)$ is not a proper cut. Hence, if $(S,L)$ is a proper cut then $\mathcal{AR}_{2}|S'\subseteq\Pi_{0}$. In particular, if $T \le S'$ and $(k_{i})_{i<\omega}$ is the increasing enumeration of $\pi''[T]$, then for each $i<\omega$, $C_{[T]}(k_{i}) = | [T] \cap p^{-1}\{k_{i}\}| = i$.

 For each $n<\omega$, let 
\begin{equation}
X_{n}= \{ x \in [S']: f( p(x) ) \ge n\}.
\end{equation}
Since ${\cal U}$ is an ultrafilter on $[T_{1}]$ we find that for each $n<\omega$, either $X_{n} \in {\cal U}$ or $[T_{1}] \setminus X_{n} \in {\cal U}$. If there exists $n< \omega$ such that $[T_{1}] \setminus X_{n} \in {\cal U}$ then $f$ would be bounded by $n$ mod $p({\cal U})$. However, this can not happen since we assumed that $f$ is not constant mod $p({\cal U})$. Hence for each $n<\omega$, $X_{n} \in {\cal U}$.

Note that $X_{0} \supseteq X_{1} \supseteq X_{2} \supseteq \dots $ is a decreasing sequence of members of ${\cal U}$. Since ${\cal U}$ is selective for $\mathcal{R}_{1}$ there exists $S''\in{\cal C}$ such that for each $n<\omega$, $[S''\setminus r_{n}(S'')] \subseteq X_{n}$.  Let $\{k_{n}: n<\omega\}$ be the increasing enumeration of $p''[S'']$. Since $\mathcal{AR}_{2}|S''\subseteq\Pi_{0}$, we find that for each $n<\omega$ and each $[S'']\cap p^{-1}\{k_{n}\}=[S''(n)]$. So for each $n<\omega$, $ [S'']\cap p^{-1}\{k_{n}\}\subseteq [ S''\setminus r_{n}(S'') ] \subseteq X_{n}$. Hence, for each $n<\omega$ and each $x\in[S'']\cap p^{-1}\{k_{n}\}$,
\begin{equation}
f(k_{n})=f(p(x)) \ge n =|[S'']\cap p^{-1}(k_{n})|= C_{[S'']}(k_{n}).
\end{equation}
Since $\{k_{n}:n<\omega\}=p''[S''] \in p({\cal U})$ we find that $[f] \ge [C_{[S'']}]$. So $[f] \in L$ as $S''\in{\cal C}$. Additionally, note that the cardinality function of any member of ${\cal C}$ is not constant mod $p({\cal U})$. Therefore the cut $(S,L)$ is the standard cut in $\omega^{\omega}/ p({\cal U})$.
\end{proof}
In the next Lemma and Theorem, ${\cal V}$ is an ultrafilter on $\omega$, and $(S,L)$ is a proper cut in $\omega^{\omega}/{\cal V}$.
\begin{lemma}
\label{CUT2}
Assume CH. If ${\cal V}$ is selective and $(S,L)$ is the standard cut in $\omega^{\omega}/ {\cal V}$ then there exists a Ramsey for $\mathcal{R}_{1}$ ultrafilter ${\cal U}$ such that $(S,L)$ is the cut associated to ${\cal U}$ and $\pi$.
\end{lemma}
\begin{proof}
Let ${\cal V}$ be selective and $(S,L)$ be the standard cut in $\omega^{\omega}/\mathcal{V}$. We will construct a $\le^{*}$-decreasing sequence $\{A_{\alpha}\in\mathcal{R}_{1}: \alpha < \omega_{1}\}$ that generates a Ramsey for $\mathcal{R}_{1}$ ultrafilter $\mathcal{U}$ on $[T_{1}]$ such that $\pi(\mathcal{U}) =\mathcal{V}$. Let $\{Z_{\alpha}: \alpha< \omega_{1}\}$ be an enumeration of the elements of $\{ Z : (\exists k) Z\subseteq \mathcal{R}_{1}(k)\}$. We impose the following requirements on the sequence, for each $\alpha<\omega_{1}$:
\begin{equation}
\mbox{Either $Z_{\alpha}$ or $\mathcal{R}_{1}(k) \setminus Z_{\alpha}$ includes $ \mathcal{R}_{1}(k) \upharpoonright A_{\alpha+1}$} \tag{$\alpha$}
\end{equation}
where $k$ is the natural number such that $Z_{\alpha}\subseteq \mathcal{R}_{1}(k).$  Since $\mathcal{R}_{1}(0)$ is in bijective correspondence with $[T_{1}]$ and $\{A_{\alpha}\in\mathcal{R}_{1}: \alpha < \omega_{1}\}$ is an almost-decreasing sequence the sequence will generate a p-point ultrafilter on $[T_{1}]$. Each $A_{\alpha}$ we be \emph{large} in the sense that $\pi''[A_{\alpha}]\in \mathcal{V}$; this suffices to guarantee that $\pi(\mathcal{U})=\mathcal{V}$. By Theorem $\ref{selectThm}$ this is enough to show that $\mathcal{U}$ is selective for $\mathcal{R}_{1}$. The conditions $(\alpha)$ for $\alpha<\omega_{1}$, guarantee that for each $k<\omega$, ${\cal U}|\mathcal{R}_{1}(k)$ is an ultrafilter on $\mathcal{R}_{1}(k)$.  By Theorem $\ref{EquivalentRamsey}$ this is enough to show that $\mathcal{U}$ is a Ramsey for $\mathcal{R}_{1}$ ultrafilter. By previous lemma this guarantees that the cut associated to $\mathcal{U}$ and $p$ in $\omega^{\omega}/\mathcal{V}$ is standard.
 
First let $A_{0} = T_{1}$. Suppose that $\beta<\omega_{1}$ and $\{A_{\alpha}: \alpha<\beta\}$ have been defined so that $A_{\alpha+1}$ satisfies the $\alpha^{th}$ condition. Assume that for each $\alpha<\beta$, $A_{\alpha}$ is large and these $A_{\alpha}$'s form a $\le^{*}$-decreasing chain. 
c
Consider the case when $\beta$ is a successor ordinal. Let $\alpha$ be the ordinal such that $\beta=\alpha+1$. Let $k$ be the natural number such that $Z_{\alpha} \subseteq \mathcal{R}_{1}(k)$. Let $(k_{m})_{m<\omega}$ be the increasing enumeration of $\pi([A_{\alpha}])\in{\cal V}$. By Theorem $\ref{finiteRamseyTheorem}$ there exists a subsequence $(k'_{m})_{m<\omega}$ such that for each $m<\omega$ there exists a set $A'_{m}\in \mathcal{R}_{1}(m)\upharpoonright A_{\alpha}(k'_{m})$ such that either $(\dagger)$ $\mathcal{R}_{1}(k)\upharpoonright A_{m}' \subseteq Z$ or ($\ddagger)$ $\mathcal{R}_{1}(k)\upharpoonright A_{m}' \subseteq \mathcal{R}_{1}(k)\setminus Z$. 

 Since ${\cal V}$ is an ultrafilter there exists a strictly increasing sequence $(m_{i})_{i<\omega}$ such that $\{k'_{m_{i}}: i<\omega\}\in \mathcal{V}$ and either for all $i<\omega$,  $\mathcal{R}_{1}(k)\upharpoonright A_{m_{i}}' \subseteq Z$ or for all $i<\omega$,  $\mathcal{R}_{1}(k)\upharpoonright A_{m_{i}}' \subseteq \mathcal{R}_{1}(k)\setminus Z$. Let $B= \bigcup_{i<\omega}A'_{m_{i}}.$ So for each $i<\omega$,
\begin{align*}
C_{[B]}(k'_{m_{i}}) &= |[B] \cap \pi^{-1}(k'_{m_{i}})|, \\
&= |[A'_{m_{i}}]|,\\
& \ge i.
\end{align*}
Consequently, the following construction of $C\in{\mathcal{R}_{1}}$ is well-defined:
\[
C = \bigcup_{i<\omega} \pi_{T(i)} ( A'_{m_{i}} ).
\]
By construction $C \le A_{\alpha}$ and either $\mathcal{R}_{1}(k) \upharpoonright C \subseteq Z$ or $\mathcal{R}_{1}(k)\upharpoonright C \subseteq \mathcal{R}_{1}(k)\setminus Z.$ Note that $C$ is large since $\pi''[C]=\{k'_{m_{i}}: i<\omega\}\in \mathcal{V}$. Let $A_{\beta}= C$, then $A_{\beta}$ is large satisfies the condition ($\alpha$).

Next consider the case when $\beta$ is a limit ordinal. Since CH holds the cofinality of $\beta$ is $\omega$. Let $\{B_{n} : n< \omega\}$ be a $\le^{*}$-cofinal sequence in $\{A_{\alpha}:\alpha<\beta\}$. So for each $i<\omega$, there exists $H_{i}\in{\cal V}$ such that for all $n\in H_{i}$, $i \le C_{B_{i}}(n) $. Without loss of generality, we may assume that $H_{0} \supseteq H_{1} \supseteq H_{2} \supseteq \dots $ Since ${\cal V}$ is selective there exist $H\in{\cal V}$, $H= \{ h_{0}, h_{1}, \dots\}$ and for all $i<\omega$, $h_{i+1} \in H_{h_{i}}$. Therefore for all $i<\omega$, $ h_{i+1} \le C_{B_{h_{i}}}(h_{i+1})$. Hence for all $i<\omega$, $i  \le C_{B_{h_{i}}}(h_{i+1})$. For each $n<\omega$, let $A'(n)= \pi_{T(n)} (B_{h_{n}}(h_{n+1}))$. Let $A_{\beta}= \bigcup_{n<\omega} A'(n)$, then $A_{\beta}\in\mathcal{R}_{1}$ is large because $\pi''[A_{\beta}] =\{h_{1},h_{2}, h_{3},
\dots\} \in{\cal V}$. Moreover, for each $n<\omega$, $A_{\beta} \le^{*} B_{h_{n}}$. So $\{ A_{\alpha}: \alpha\le \beta\}$ is a $\le^{*}$-decreasing sequence of large sets.

This completes the inductive construction of $\{A_{\beta}: \beta < \aleph_{1}\}$ and thus the proof of the lemma.
\end{proof}

\begin{proof}[\emph{Proof of Theorem $\ref{MainResult1}$}]
Lemma $\ref{CUT2}$ and $\ref{selectThm}$ show that necessity holds and Lemma $\ref{CUT1}$ shows that sufficiency holds. \end{proof}

\section{ The cuts obtained from $\pi_{T(i)}$ and Ramsey for $\mathcal{R}_{1}$ ultrafilters} \label{section5}
If ${\cal U}_{1}$ is a Ramsey for $\mathcal{R}_{1}$ ultrafilter generated by a generic subset of $(\mathcal{R}_{1},\le^{*})$ and $i<\omega$, then we let $\mathcal{Y}_{i+1}$ denote the ultrafilter $\mathcal{U}_{1}|\mathcal{R}_{1}(i)$. In this section,  we show that for any Rudin-Keisler mapping between any two p-points within the Tukey type of $\mathcal{U}_{1}$ the only cut obtainable is the standard cut. Notice that for each $i<j<\omega$ and each Ramsey for $\mathcal{R}_{1}$ ultrafilter $\mathcal{U}_{1}$,
\begin{equation}
\pi_{T(i)}( \mathcal{U}_{1}|\mathcal{R}_{1}(j) ) =\mathcal{U}_{1}| \mathcal{R}_{1}(i).
\end{equation}
Additionally, notice that if $Z \in \mathcal{U}_{1}| \mathcal{R}_{1}(j)$ then the cardinality function of $Z$ with respect to $\pi_{T(i)}$ is defined as
\begin{equation}
C_{Z}(p) = | Z \cap \pi_{T(i)}^{-1}(p) |, \ \mbox{ for } p\in \mathcal{R}_{1}(i).
\end{equation} A simple counting argument shows that for each $S\in \mathcal{R}_{1}$, each $n<\omega$ and each $p\in \mathcal{R}_{1}(i)\upharpoonright S(n)$,
\begin{equation}
C_{\mathcal{R}_{1}(j)\upharpoonright S} ( p ) = | \mathcal{R}_{1}(j)\upharpoonright S \cap \pi^{-1}_{T(i)}(p) | \le { n -i  \choose j-i}.
\end{equation}
Next we outline the basic definitions of the Tukey theory of ultrafilters. The Tukey types of ultrafilters have been studied by many authors (see \cite{Isbell}, \cite{DavidMilovich}, \cite{TukeyTypesDT}, \cite{Ramsey-Class}, \cite{Ramsey-Class2} and \cite{RaghavanTod}). For a survey of the area see \cite{SurveyDob} by Dobrinen.
\begin{definition}
Suppose that $\mathcal{U}$ and $\mathcal{V}$ are ultrafilters. A function $f$ from $\mathcal{U}$ to $\mathcal{V}$ is \emph{Tukey} if every cofinal subset of $(\mathcal{U},\supseteq)$ is mapped by $f$ to a cofinal subset of $(\mathcal{V}, \supseteq)$. We say that $\mathcal{V}$ is \emph{Tukey reducible to $\mathcal{U}$} and write $\mathcal{V} \le_{T} \mathcal{U}$ if there exists a Tukey map $f:\mathcal{U} \rightarrow \mathcal{V}$. If $\mathcal{U} \le_{T}\mathcal{V}$ and $\mathcal{V} \le_{T}\mathcal{U}$ then we write $\mathcal{V}\equiv_{T} \mathcal{U}$ and say that $\mathcal{U}$ and $\mathcal{V}$ are \emph{Tukey equivalent.} The relation $\equiv_{T}$ is an equivalence relation and $\le_{T}$ is a partial order on its equivalence classes.  The equivalence class are also called \emph{Tukey types} or \emph{Tukey degrees}.
\end{definition}
The Tukey reducibility relation is a generalization of the Rudin-Keisler reducibility relation. If $h(\mathcal{U})=\mathcal{V}$ then the map sending $X\in\mathcal{U}$ to $h''X \in\mathcal{V}$ is Tukey. So  if $\mathcal{V} \le _{RK} \mathcal{U}$, then $\mathcal{V}\le_{T} \mathcal{U}$. This leads to the following question: For a given ultrafilter $\mathcal{U}$, what is the structure of the Rudin-Keisler ordering within the Tukey type of $\mathcal{U}$?

Dobrinen and Todorcevic in \cite{Ramsey-Class}, have given an answer to this question if $\mathcal{U}$ is a Ramsey for $\mathcal{R}_{1}$ ultrafilter on $[T_{1}]$ generated by a generic subset of $(\mathcal{R}_{1}, \le^{*})$. Furthermore, Dobrinen and Todorcevic describe the Rudin-Keisler structure of the p-points within the Tukey type of a Ramsey for $\mathcal{R}_{1}$ ultrafilter generated by a generic subset of $(\mathcal{R}_{1}, \le^{*})$. In particular, the Tukey type of such an ultrafilter $\mathcal{U}_{1}$ consists of a strictly increasing chain of rapid p-points of order type $\omega$: $\mathcal{Y}_{1} <_{RK} \mathcal{Y}_{2} <_{RK} \dots$ where $\mathcal{Y}_{i+1}= \mathcal{U}_{1}| \mathcal{R}_{1}(i)$ for $i<\omega$. The next proof is the main result of this section.

\begin{proof}[\emph{Proof of Theorem $\ref{MainResult2}$}] Suppose that $\mathcal{U}_{1}$ is a Ramsey for $\mathcal{R}_{1}$ ultrafilter on $[T_{1}]$ generated by a generic subset of $(\mathcal{R}_{1}, \le^{*})$. Let $\mathcal{Y}$ be a p-point ultafilter in the Tukey-type of $\mathcal{U}_{1}$ and $g$ be a map from the base of $\mathcal{Y}$ to $\omega$. Let $(S,L)$ be the cut associated to $\mathcal{Y}$ and $g$. By Theorem 5.10 and Example 5.17 from \cite{Ramsey-Class} there exists $i,j<\omega$ such that $i\le j$, $\mathcal{Y} \cong \mathcal{Y}_{j+1}$ and $g(\mathcal{Y}) \cong \mathcal{Y}_{i+1}$. Notice that since $(S,L)$ is not proper we have $i<j$. By Remark \ref{IsomorphismRemark}, it is enough to prove the this theorem in the case when $\mathcal{Y}=\mathcal{Y}_{j+1}$ and $g(\mathcal{Y})=\mathcal{Y}_{i+1}$. In particular, we may assume without loss of generality that $g:\mathcal{R}_{1}(j) \rightarrow \mathcal{R}_{1}(i)$. 

Since $\mathcal{U}_{1}$ is generated by a generic subset of $(\mathcal{R}_{1},\le^{*})$, we find that Theorem 4.25 and Proposition 5.8 in \cite{Ramsey-Class} imply that there exists $T\in \mathcal{C}$ and $i'\le i$ such that for all $p,q\in \mathcal{R}_{1}(i)\upharpoonright T$, $g(p)=g(q)$ if and only if $\pi_{T(i')}(p) = \pi_{T(i')}(q)$. So if $T'\in\mathcal{C}$, then for each $n<\omega$ and each $p\in \mathcal{R}_{1}(i)\upharpoonright T(n)$,
\begin{equation}\label{CardEquation}
C_{\mathcal{R}_{1}(i)\upharpoonright (T' \cap T)}(p)=|\mathcal{R}_{1}(i)\upharpoonright (T' \cap T)\cap g^{-1}(p)| = |\mathcal{R}_{1}(i)\upharpoonright (T' \cap T) \cap \pi_{T(i')}^{-1}(p)|\le {n-i' \choose i - i'}.
\end{equation}

Let $f:\mathcal{R}_{1}(i) \rightarrow \omega$ be given and suppose that $f$ is not constant mod ${\cal U}|\mathcal{R}_{1}(i)$. For each $n<\omega$, let 
\begin{equation}
X_{n}= \{ p \in \mathcal{R}_{1}(i) : f(p) \ge{ n - i' \choose i- i'} \}.
\end{equation}
Since $\mathcal{U}_{1}|\mathcal{R}_{1}(i)$ is an ultrafilter on $\mathcal{R}_{1}(i)$ we find that for each $n<\omega$, either $X_{n} \in \mathcal{U}_{1}|\mathcal{R}_{1}(i)$ or $\mathcal{R}_{1}(i) \setminus X_{n} \in \mathcal{U}_{1}|\mathcal{R}_{1}(i)$. If there exists $n< \omega$ such that $\mathcal{R}_{1}(i)  \setminus X_{n} \in \mathcal{U}_{1}|\mathcal{R}_{1}(i) $ then $f$ would be bounded by $n$ mod $\mathcal{U}_{1}|\mathcal{R}_{1}(i) $. However, this can not happen since we assumed that $f$ is not constant mod $\mathcal{U}_{1}|\mathcal{R}_{1}(i) $. Hence for each $n<\omega$, $X_{n} \in \mathcal{U}_{1}|\mathcal{R}_{1}(i)$.

Note that $X_{0} \supseteq X_{1} \supseteq X_{2} \supseteq \dots $ is a decreasing sequence of members of $\mathcal{U}_{1}|\mathcal{R}_{1}(i)$. Since $\mathcal{U}_{1}$ is selective for $\mathcal{R}_{1}$ there exists $S'\in{\cal C}$ such that for each $n<\omega$, $\mathcal{R}_{1}(i)\upharpoonright S'\setminus r_{n}(S') \subseteq X_{n}$. For each $n<\omega$ and each $p \in \mathcal{R}_{1}(i)\upharpoonright S'(n)$, $p \in \mathcal{R}_{1}(i) \upharpoonright (S'\setminus r_{n}(S')) \subseteq X_{n}$. So ($\ref{CardEquation}$) implies that for each $n<\omega$ and each $p\in\mathcal{R}_{1}(i)\upharpoonright S'(n)$,
\begin{equation}
f(p) \ge { n- i' \choose i-i'} \ge C_{\mathcal{R}_{1}(j)\upharpoonright (T' \cap  S')}(p).
\end{equation}

So $\mathcal{R}_{1}(i) \upharpoonright(T \cap S') \in \mathcal{U}_{1}|\mathcal{R}_{1}(i)$ and $[f] \ge [C_{\mathcal{R}_{1}(i)\upharpoonright(T\cap S')}]$. So $[f] \in L$ as $\mathcal{R}_{1}(i)\upharpoonright(T\cap S')\in\mathcal{Y}_{i+1}$. Additionally, note that the cardinality function of any member of $\mathcal{Y}_{j+1}$ is not constant mod $g(\mathcal{Y}_{j+1})=\mathcal{U}_{1}|\mathcal{R}_{1}(i) $. Therefore the cut $(S,L)$ is the standard cut in $\omega^{\mathcal{R}_{1}(i)}/  \mathcal{U}_{1}|\mathcal{R}_{1}(i)$.
\end{proof}

\section{Conclusion} \label{section6}
In this section we use the two main results to prove that, under CH, certain special ultrafilters exists. Additionally, we ask some questions about the types of cuts that can be obtained from similarly defined topological Ramsey spaces. The two main results of this paper only associate the standard cut to a given ultrafilter mapping. However the results of Blass from \cite{BlassCut} only require the lower half of the cuts to be closed under certain operations. These differences allow us to show that, under CH, certain special ultrafilters exists which are not Ramsey for $\mathcal{R}_{1}$.
\begin{corollary}
Assume CH. There exists a p-point ultrafilter on $[T_{1}]$ which is neither weakly Ramsey nor Ramsey for $\mathcal{R}_{1}$.
\end{corollary}
\begin{proof}
Recall that under CH, p-point ultrafilters exist. Let $\mathcal{U}$ be a p-point ultrafilter on $\omega$, $[f]$ be a nonstandard element of $\omega^{\omega}/\mathcal{U}$ and  
\begin{equation}
S =  \bigcup_{n<\omega}\{ [g] \in \omega^{\omega}/\mathcal{U}:  [g]\le n\cdot[f]\}.
\end{equation}
Notice that $S$ is closed under addition and contains the standard part of $\omega^{\omega}/\mathcal{U}$. Let $L$ be the complement of $S$. The condition that $L$ is nonempty and every countable subset of $L$ has a lower bound in L, will be automatically satisfied if $S$ has a
countable cofinal subset, because $\omega^{\omega}/\mathcal{U}$ is countably saturated (see \cite{KeislerSat}). So $(S,L)$ is a proper cut such that $S$ is closed under addition and every countable subset of $L$ has a lower bound in $L$. By Theorem $\ref{cut-thm2}$ there is a p-point ultrafilter $\mathcal{V}$ and a Rudin-Keisler map $p$ such that $p(\mathcal{V})=\mathcal{U}$ and $(S,L)$ is the cut associated to $p$ and $\mathcal{V}$. Since $(S,L)$ is not the standard cut Theorem $\ref{MainResult1}$ implies that $\mathcal{V}$ is not Ramsey for $\mathcal{R}_{1}$. Since $(S,L)$ is not closed under exponentiation Theorem $\ref{cut-thm4}$ implies that $\mathcal{V}$ is not weakly-Ramsey.
\end{proof}
We remind the reader of the recursive definition of \emph{iterated exponentials}. For each natural number $n$ and $i$ we let,
\begin{equation}\begin{cases}
{}^{0}n=1,  \\
{}^{i+1} n  = n^{{}^{i}n }.
\end{cases} \end{equation}

For example, ${}^{3}2 = 2^{2^{2}}$ and ${}^{2}6 = 6^{6}$. Note that if $k, n$ and $m$ are natural numbers, then $({}^{n}k)^{{}^{m} k} \le {}^{n+m}k$.

\begin{corollary}
Assume CH. There exists a weakly Ramsey ultrafilter on $[T_{1}]$ which is not Ramsey for $\mathcal{R}_{1}$.
\end{corollary}
\begin{proof}
Recall that under CH, selective ultrafilters exist. Let $\mathcal{U}$ be a selective ultrafilter on $\omega$, $[f]$ be a nonstandard element of $\omega^{\omega}/\mathcal{U}$ and  
\begin{equation}
S = \bigcup_{n<\omega} \{ [g] \in \omega^{\omega}/\mathcal{U}:  [g]\le {}^{n}[f]\}.
\end{equation}
If $[g]$ and $[h]$ are in $S$ then there exists natural numbers $n$ and $m$ such that $[g]\le {}^{n}[f]$ and $[h]\le {}^{m}[f]$. So $[g]^{[h]} \le {}^{n+m}[f]$. Therefore $S$ is closed under exponentiation. Additionally $S$ contains the standard part of $\omega^{\omega}/\mathcal{U}$. Let $L$ be the complement of $S$. The condition that $L$ is nonempty and every countable subset of $L$ has a lower bound in L, will be automatically satisfied if $S$ has a countable cofinal subset, because $\omega^{\omega}/\mathcal{U}$ is countably saturated (see \cite{KeislerSat}). So $(S,L)$ is a proper cut such that $S$ is closed under exponentiation and every countable subset of $L$ has a lower bound in $L$.  By Theorem $\ref{cut-thm4}$ there is a weakly Ramsey ultrafilter $\mathcal{V}$ and a Rudin-Keisler map $p$ such that $p(\mathcal{V})=\mathcal{U}$ and $(S,L)$ is the cut associated to $p$ and $\mathcal{V}$. Since $(S,L)$ is not the standard cut Theorem $\ref{MainResult1}$ implies that $\mathcal{V}$ is not Ramsey for $\mathcal{R}_{1}$.
\end{proof}
In \cite{SelectiveButNotRamsey} Trujillo has shown that assuming CH there exists a selective for $\mathcal{R}_{1}$ ultrafilter which is not Ramsey for $\mathcal{R}_{1}$. Using a similar proof to that in Theorem $\ref{MainResult1}$ it is possible to characterize the cuts obtainable from a selective for $\mathcal{R}_{1}$ ultrafilter and the map $\pi$; they are exactly the standard cuts. However, it is unclear if there is a selective for $\mathcal{R}_{1}$ ultrafilter $\mathcal{U}$ and a Rudin-Keisler map $p$ such that the cut associated to $p$ and $\mathcal{U}$ is not standard. This motivates the following question:
\begin{question}
Can the notions of selective for $\mathcal{R}_{1}$ and Ramsey for $\mathcal{R}_{1}$ be distinguished by characterizing the Dedekind cuts obtainable from selective for $\mathcal{R}_{1}$ ultrafilters on $[T_{1}]$?
\end{question}
If it is shown that the only cuts obtainable are standard then this method will not distinguish between the two notions. However if there exists a selective for $\mathcal{R}_{1}$ ultrafilter and a Rudin-Keisler map $p$ such that $(S,L)$ is not standard then $\mathcal{U}$ will not be Ramsey for $\mathcal{R}_{1}$ and the two notions will be distinguished. 

Dobrinen and Todorcevic in \cite{Ramsey-Class2} have defined generalizations of the space $\mathcal{R}_{1}$ for $1< \alpha<\omega_{1}$. The spaces are built from trees $T_{\alpha}$ in much the same way that $\mathcal{R}_{1}$ is built from $T_{1}$. Trujillo \cite{SelectiveButNotRamsey} has shown that for each positive integer $n$, under CH, there are selective for $\mathcal{R}_{n}$ ultrafilters on $[T_{n}]$ which are not Ramsey for $\mathcal{R}_{n}$. However it is still unknown if, under CH, there are selective for $\mathcal{R}_{\alpha}$ ultrafilters on $[T_{\alpha}]$ which are not Ramsey for $\mathcal{R}_{\alpha}$ for $\omega\le \alpha <\omega_{1}$. The methods used in \cite{SelectiveButNotRamsey} fail for infinite $\alpha$ as the trees $T_{\alpha}$ have infinite height.

\begin{question} Assume that $1<\alpha < \omega_{1}$.
Can the notions of selective for $\mathcal{R}_{\alpha}$ and Ramsey for $\mathcal{R}_{\alpha}$ be distinguished by characterizing the Dedekind cuts obtainable from selective and Ramsey for $\mathcal{R}_{\alpha}$ ultrafilters on $[T_{\alpha}]$?
\end{question}

We may also prove similar existence results using the second main result, Theorem $\ref{MainResult2}$.
\begin{corollary}
Assume CH holds and $\mathcal{U}_{1}$ is a Ramsey for $\mathcal{R}_{1}$ ultrafilter on $[T_{1}]$ generated by a generic subset of $(\mathcal{R}_{1}, \le^{*})$. There exists a weakly Ramsey ultrafilter $\mathcal{V}$ on $[T_{1}]$ such that $\mathcal{V}\not \le_{T} \mathcal{U}_{1}.$
\end{corollary}
\begin{proof}
Notice that $\pi(\mathcal{U}_{1})$ is a selective ultrafilter by Theorem \ref{selectThm}. Let $[f]$ be a nonstandard element of $\omega^{\omega}/\pi(\mathcal{U}_{1})$ and
\begin{equation}
S= \bigcup_{n<\omega} \{ [g] \in \omega^{\omega}/\pi(\mathcal{U}_{1}):  [g]\le {}^{n}[f]\}.
\end{equation}
$S$ is closed under exponentiation. Additionally $S$ contain the standard part of $\omega^{\omega}/\pi(\mathcal{U})$. Let $L$ be the complement of $S$. The condition that $L$ is nonempty and every countable subset of $L$ has a lower bound in $L$, is automatically satisfied if $S$ has a countable cofinal subset, because $\omega^{\omega}/\pi(\mathcal{U}_{1})$ is countably saturated (see \cite{KeislerSat}). So $(S,L)$ is a proper cut such that $S$ is closed under exponentiation and every countable subset of $L$ has a lower bound in $L$.  By Theorem $\ref{cut-thm4}$ there is a weakly Ramsey ultrafilter $\mathcal{V}$ and a Rudin-Keisler map $p$ such that $p(\mathcal{V})=\pi(\mathcal{U}_{1})$ and $(S,L)$ is the cut associated to $p$ and $\mathcal{V}$. Since $(S,L)$ is not the standard cut Theorem $\ref{MainResult2}$ implies that $\mathcal{V}$ is not Tukey-reducible to $\mathcal{U}_{1}$.
\end{proof}

\begin{corollary}
Assume CH holds and $\mathcal{U}_{1}$ is a Ramsey for $\mathcal{R}_{1}$ ultrafilter on $[T_{1}]$ generated by a generic subset of $(\mathcal{R}_{1}, \le^{*})$. There  exists a p-point ultrafilter $\mathcal{W}$ on $[T_{1}]$ which is not weakly Ramsey such that $\mathcal{W}>_{T} \mathcal{U}_{1}$.
\end{corollary}
\begin{proof}
Notice that $\pi_{T(1)}(\mathcal{U}_{1})$ is a p-point ultrafilter by Theorem \ref{selectThm}. Let $[f]$ be a nonstandard element of $\omega^{\omega}/\pi_{T(1)}(\mathcal{U}_{1})$ and
\begin{equation}
S = \bigcup_{n<\omega} \{ [g] \in \omega^{\omega}/\pi_{T(1)}(\mathcal{U}_{1}):  [g]\le n\cdot[f]\} \mbox{ and }
\end{equation}
$S$ is under addition. Additionally $S$ contain the standard part of $\omega^{\omega}/\pi_{T(1)}(\mathcal{U})$. Let $L$ be the complement of $S$. The condition that $L$ is nonempty and every countable subset of $L$ has a lower bound in $L$, is automatically satisfied if $S$ has a countable cofinal subset, because $\omega^{\omega}/\pi_{T(1)}(\mathcal{U}_{1})$ is countably saturated (see \cite{KeislerSat}). So $(S,L)$ is a proper cut such that $S$ is closed under addition and every countable subset of $L$ has a lower bound in $L$.  By Theorem $\ref{cut-thm2}$ there is a p-point ultrafilter $\mathcal{W}$ and a Rudin-Keisler map $p$ such that $p(\mathcal{W})=\pi_{T(1)}(\mathcal{U}_{1})$ and $(S,L)$ is the cut associated to $p$ and $\mathcal{W}$. Since $(S,L)$ is not the standard cut Theorem $\ref{MainResult2}$ implies that $\mathcal{W}$ is not Tukey-reducible to $\mathcal{U}_{1}$. On the other hand, $\pi_{T(1)}(\mathcal{U}_{1})$ is Tukey equivalent to $\mathcal{U}_{1}$. Therefore $\mathcal{W}>_{T} \mathcal{V}$.
\end{proof}
\begin{remark}
Notice that the previous corollary shows that the converse of Theorem $\ref{MainResult2}$, under CH, does not hold. In particular there is a p-point $\mathcal{W}$ such that there is no p-point in the Tukey-type of $\mathcal{U}_{1}$ which gives the standard cut in $\omega^{\omega}/\mathcal{W}$ since otherwise $\mathcal{W}$ would be Tukey reducible to $\mathcal{U}_{1}$.
\end{remark}
This leads naturally to the next question.
\begin{question}
Is it possible to strengthen the conclusion of Theorem $\ref{MainResult2}$ so that, under CH, a converse to the modified theorem holds?
\end{question}
The previous remark shows that at the very least one needs to assume that the not only is $\mathcal{V}$ a p-point but it is also Tukey-reducible to $\mathcal{U}_{1}$. We conclude this paper with a more general question of Dobrinen from \cite{SelectiveButNotRamsey} which has motivated much of the work in this paper. 

\begin{question} For a given topological Ramsey space $\mathcal{R}$, are the notions of selective for $\mathcal{R}$ and Ramsey for $\mathcal{R}$ equivalent?
\end{question}

\end{document}